\documentclass[12pt,a4paper]{amsart}
\usepackage{amsmath,amsfonts,amssymb,amscd,amsthm,amsrefs,tikz,comment}
\usepackage[shortlabels]{enumitem}
\usetikzlibrary{arrows, arrows.meta}

\usepackage{color, soul}
\definecolor{gr}{rgb}{0.7, 1, 0.7}
\definecolor{rr}{rgb}{1, 0.7, 0.7}

\usepackage{latexsym}
\usepackage{graphicx}

\theoremstyle{plain} 
\newtheorem{theorem}{Theorem}[section]

\newtheorem{thmx}{Theorem}

\newtheorem{lemma}[theorem]{Lemma}
\newtheorem{corollary}[theorem]{Corollary}

\newtheorem{proposition}[theorem]{Proposition}

\theoremstyle{definition} 

\theoremstyle{remark} 
\newtheorem{remark}[theorem]{Remark}

\renewcommand{\mathfrak}{\mathbf}

\newcommand{\ignore}[1]{}

\newcommand{\bbC}{\mathbb{C}}

\newcommand{\bbN}{\mathbb{N}}
\newcommand{\bbR}{\mathbb{R}}
\newcommand{\bbZ}{\mathbb{Z}}

\newcommand{\bbD}{\mathbb{D}}

\newcommand{\tl}{\tilde}

\newcommand{\dist}{\operatorname{dist}}

\newcommand{\M}{\mathbb M}

\newcommand{\mbif}{\mu_{\mathrm{bif}}}

\title[]{Accumulation set of critical points of the multipliers in the quadratic family}
\author{Tanya Firsova}

\address{Kansas State University, Manhattan, KS, USA and National Research University Higher School of Economics, Russian Federation}
\email{tanyaf@math.ksu.edu}

\author{Igors Gorbovickis}
\address{Jacobs University, Bremen, Germany}
\email{i.gorbovickis@jacobs-university.de}

\subjclass[2010]{}
\date{\today}

\begin{document}
\begin{abstract}
A parameter $c_0\in\bbC$ in the family of quadratic polynomials $f_c(z)=z^2+c$ is a \textit{critical point of a period $n$ multiplier}, if the map $f_{c_0}$ has a periodic orbit of period $n$, whose multiplier, viewed as a locally analytic function of $c$, has a vanishing derivative at $c=c_0$. We study the accumulation set $\mathcal X$ of the critical points of the multipliers, as $n\to\infty$. 
This study complements the equidistribution result for the critical points of the multipliers that was previously obtained by the authors.
In particular, in the current paper we prove that the accumulation set $\mathcal X$ is bounded, path connected and contains the Mandelbrot set as a proper subset. 
We also provide a necessary and sufficient condition for a parameter outside of the Mandelbrot set to be contained in the accumulation set $\mathcal X$ and show that this condition is satisfied for an open set of parameters. Our condition is similar in flavor to one of the conditions that define the Mandelbrot set.
As an application, we get that the function that sends $c$ to the  Hausdorff dimension of $f_c$, does not have critical points outside of the accumulation set $\mathcal X$.
\end{abstract}
\maketitle

\section{Introduction}

Consider the family of quadratic polynomials
$$
f_c(z)=z^2+c,\qquad c\in\bbC.
$$
We say that a parameter $c_0\in\bbC$ is a \textit{critical point of a period $n$ multiplier}, if the map $f_{c_0}$ has a periodic orbit of period $n$, whose multiplier, viewed as a locally analytic function of $c$, has a vanishing derivative at $c=c_0$. The study of critical points of the multipliers is motivated by the problem of understanding the geometry of hyperbolic components of the Mandelbrot set.

As it was observed by D.~Sullivan and A.~Douady and J.~Hubbard~\cite{Douady_Hubbard_Orsay_2}, the argument of quasiconformal surgery implies that the multipliers of periodic orbits, viewed as analytic functions of the parameter $c$, are Riemann mappings of the corresponding hyperbolic components of the Mandelbrot set. Existence of analytic extensions of the inverse branches of these Riemann mappings to larger domains can be helpful in estimating 
the geometry of the hyperbolic components as well as the sizes of some limbs of the Mandelbrot set~\cites{Levin_2009,Levin_2011} (see also~\cite{Dezotti_thesis}). 
Critical values of the multipliers are the only obstructions to existence of these analytic extensions.

It is of special interest to obtain uniform bounds on the shapes of hyperbolic components within renormalization cascades. In particular, this motivates the study of the asymptotic behavior of the critical points of period $n$ multipliers as $n\to\infty$. In~\cite{Firsova_Gor_equi} the current authors approached this questions from the statistical point of view and proved that the critical points of the period $n$ multipliers equidistribute on the boundary of the Mandelbrot set, as $n\to\infty$.

More specifically, for each $n\in\bbN$, let $X_n$ be the set of all parameters $c\in\bbC$ that are critical points of a period $n$ multiplier (counted with multiplicities). 
Let $\M\subset\bbC$ denote the Mandelbrot set and let $\mbif$ be its equilibrium measure (or the bifurcation measure of the quadratic family $\{f_c\}$). Let $\delta_x$ denote the $\delta$-measure at $x\in\bbC$. Then

\begin{theorem}\cite{Firsova_Gor_equi}\label{equidistr_theorem}
The sequence of probability measures
$$
\frac{1}{\# X_n}\sum_{x\in X_n}\delta_x
$$
converges to the equilibrium measure $\mbif$ in the weak sense of measures on $\bbC$, as $n\to\infty$. 
\end{theorem}

At the same time, it was shown in~\cite{Belova_Gorbovickis} that $0$ is a critical point of infinitely many multipliers of different periodic orbits, hence, since $0\not\in\partial\M=\mathrm{supp}(\mbif)$, this implies that as the period $n$ grows to infinity, the critical points of period $n$ multipliers accumulate on some set $\mathcal X\subset\bbC$ that is strictly greater than the support of the bifurcation measure $\mbif$.

The purpose of the current paper is to study this accumulation set $\mathcal X$ which can formally be defined as
$$
\mathcal X:= \bigcap_{k=1}^\infty \left(\,\overline{\bigcup_{n=k}^\infty X_n} \,\right).
$$
We note that the study of the accumulation set $\mathcal X$ complements the statistical approach of Theorem~\ref{equidistr_theorem} in the attempt to understand asymptotic behavior of the critical points of the multipliers.

For the portion of the set $\mathcal X$ lying outside of the Mandelbrot set $\M$, the following theorem was proved by the current authors in~\cite{Firsova_Gor_equi}:

\begin{theorem}\cite{Firsova_Gor_equi}\label{XoutsideM_theorem}
If $c\in\bbC\setminus\M$ is a critical point of some multiplier, then $c\in\mathcal X$. Equivalently, the following identity holds:
$$
\overline{\bigcup_{n=1}^\infty (X_n\setminus\M)} =\mathcal X\setminus\M.
$$
\end{theorem}
It is important to mention that it does not follow from Theorem~\ref{XoutsideM_theorem} that there exist critical points of the multipliers outside of the Mandelbrot set $\M$ and that the set $\mathcal X\setminus\M$ is non-empty, although numerical computations from~\cite{Belova_Gorbovickis} suggest that this is the case.

The first result of this paper is the following:
\begin{thmx}\label{main_theorem_1}
The accumulation set $\mathcal X$ is bounded, path connected and contains the Mandelbrot set $\M$. Furthermore, the set $\mathcal X\setminus\M$ is nonempty and has a nonempty interior.
\end{thmx}

\begin{figure}\label{X_picture}
\begin{center}
\includegraphics[width=0.8\textwidth]{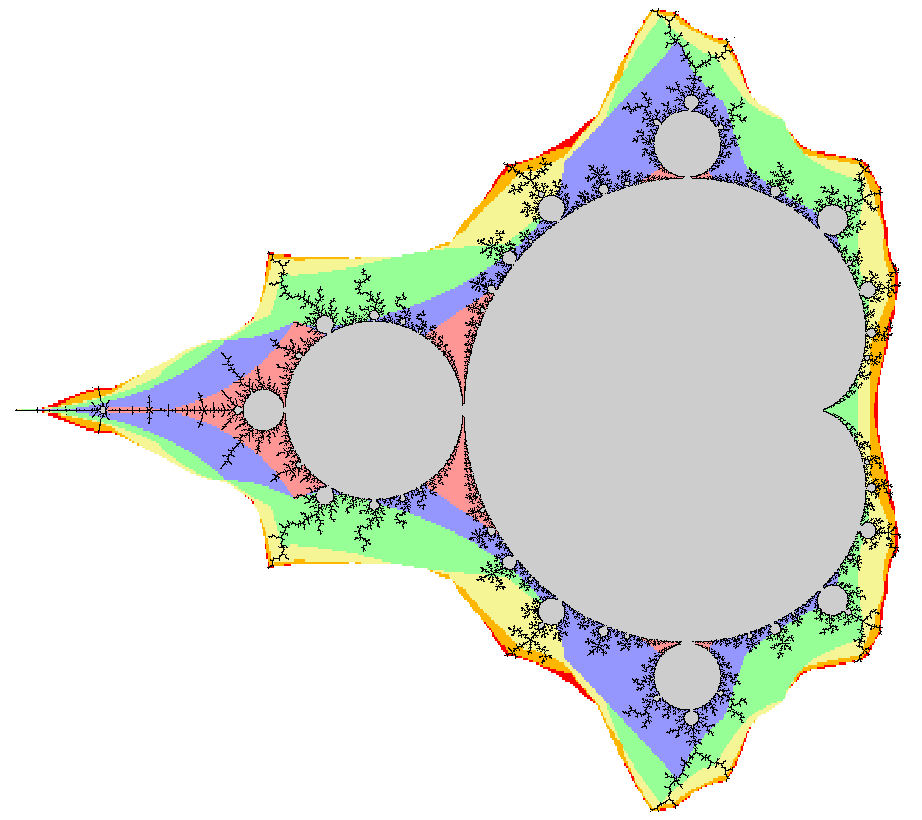}
\caption{The set $\mathcal X$ is numerically approximated by the union of the Mandelbrot set and the colored regions. The algorithm for the construction of this picture, as well as the meaning of the colors are explained in Appendix~\ref{Pic_section}.}
\end{center}
\end{figure}

Figure~\ref{X_picture} provides a numerical approximation of the accumulation set~$\mathcal X$.

We need a few more definitions in order to state our next result. For a periodic orbit $\mathcal O$ of some map $f_c$, let $|\mathcal O|$ stand for its period (i.e., the number of distinct points in it).

We recall that a periodic orbit is called \textit{primitive parabolic} if its multiplier is equal to $1$. As discussed in~\cite{Firsova_Gor_equi}, for every $c_0\in\bbC$ and every periodic orbit $\mathcal O$ of $f_{c_0}$ that is not primitive parabolic, the multiplier of this periodic orbit can be viewed as a locally analytic function of the parameter $c$ in the neighborhood of $c_0$. We denote this function by $\rho_{\mathcal O}$. 
If in addition to that, $\rho_{\mathcal O}(c_0)\neq 0$, one can consider a locally analytic function $\nu_{\mathcal O}$, defined in a neighborhood of $c_0$ by the formula
\begin{equation}\label{def_nu_eq}
\nu_{\mathcal O}(c):= \frac{\rho'_{\mathcal O}(c)}{|\mathcal O|\,\rho_{\mathcal O}(c)}.
\end{equation}

For each $c\in\bbC$, 
let $\Omega_c$ denote the set of all repelling periodic orbits of the map $f_c$. In particular, the locally analytic maps $\nu_{\mathcal O}$ are defined for all $\mathcal O\in\Omega_c$ in corresponding neighborhoods of the parameter $c$.

For each $c\in\bbC$, we consider the set $\mathcal Y_c\subset\bbC$, defined by
$$
\mathcal Y_c:= \overline{\left\{\nu_{\mathcal O}(c) \mid \mathcal O\in\Omega_c\right\}}.
$$

Our second result is the following:

\begin{thmx}\label{main_theorem_2}
The following two properties hold:
\begin{enumerate}[(i)]
\item\label{Y_convex_property} For every parameter $c\in\bbC\setminus\{-2\}$, the set $\mathcal Y_c$ is convex; 
for $c=-2$, the set $\mathcal Y_{-2}$ is the union of a convex set and the point~$-\frac{1}{6}$. 
\item\label{YX_dichotomy} For every parameter $c\in\bbC\setminus\M$, the set $\mathcal Y_c$ is bounded. A parameter $c\in\bbC\setminus\M$ belongs to $\mathcal X$, if and only if $0\in\mathcal Y_c$.
\end{enumerate}
\end{thmx}

We note that the relation between the sets $\mathcal Y_c$ and $\mathcal X$, described in part~\ref{YX_dichotomy} of Theorem~\ref{main_theorem_2}, resembles the relation between the filled Julia and the Mandelbrot sets, namely that $c\in\M$, if and only if $0$ belongs to the filled Julia set $K_c$ of $f_c$.

As an application of our results and the results of~\cite{He_Nie_2020}, we deduce that the Hausdorff dimension function cannot have critical points outside of the accumulation set $\mathcal X$. More specifically, let $\delta\colon\bbC\to\bbR$ be the function that assigns to each parameter $c\in\bbC$ the Hausdorff dimension of the Julia set of $f_c$. It is known that the function $\delta$ is real-analytic in each hyperbolic component~\cite{Bowen_79} (including the complement of the Mandelbrot set).

\begin{corollary}
The Hausdorff dimension function $\delta$ has no critical points in $\bbC\setminus\mathcal X$.
\end{corollary}
\begin{proof}
According to Theorem~\ref{main_theorem_1}, the set $\bbC\setminus\mathcal X$ is contained in the complement of the Mandelbrot set $\M$, hence, $\delta$ is real-analytic on $\bbC\setminus\mathcal X$. Then Theorem~\ref{main_theorem_2} together with~\cite[Theorem~1.3]{He_Nie_2020} implies that $\delta$ has no critical points in $\bbC\setminus\mathcal X$.
\end{proof}

\subsection*{Open questions}
Finally, we list some further questions that can be addressed in the study of the geometry of the accumulation set $\mathcal X$ and the sets $\mathcal Y_c$.

\begin{enumerate}
\item Is the set $\mathcal X$ simply connected?
\item Does the boundary of the set $\mathcal X$ possess any kind of self-similarity? Is the Hausdorff dimension of $\partial\mathcal X$ equal to~$1$ or is it strictly greater than~$1$?
\item For which $c\in\bbC$ are the sets $\mathcal Y_c$ polygonal? How are the points of the finite sets $Y_{c,n}=\{\nu_{\mathcal O}(c)\mid \mathcal O\in\Omega_c, |\mathcal O|=n\}$ distributed inside $\mathcal Y_c$ as $n\to\infty$?
\item What can we say about the geometry of the sets $\mathcal Y_c$, when $c\in\partial\M$? Are these sets always unbounded? 
\end{enumerate}

\subsection{Acknowledgements} Research of T. Firsova was supported in part by NSF grant DMS–1505342 , and by Laboratory of Dynamical Systems and Applications NRU HSE, grant of the Ministry of science and higher education of the RF ag. N 075-15-2019-1931.

\section{On averaging several periodic orbits}

In this section we state and prove the so called Averaging Lemma which is the key component of the proofs of Theorem~\ref{main_theorem_1} and Theorem~\ref{main_theorem_2}.

\begin{lemma}[\textbf{Averaging Lemma}]\label{averaging_lemma}
	For any real $\alpha\in[0,1]$, a complex parameter $c_0\in\bbC$ and any two distinct repelling periodic orbits $\mathcal O_1$ and $\mathcal O_2$ of $f_{c_0}$, such that if $c_0=-2$, then neither of the orbits $\mathcal O_1, \mathcal O_2$ is the fixed point $z=2$, the following holds: there exist a neighborhood $U$ of $c_0$ and a sequence of distinct repelling periodic orbits $\{\mathcal O_j\}_{j=3}^\infty$ of $f_{c_0}$, such that the maps $\nu_{\mathcal O_j}$ are defined and analytic in $U$, for all $j\in\bbN$, and the sequence of maps $\{\nu_{\mathcal O_j}\}_{j=3}^\infty$ converges to $\alpha\nu_{\mathcal O_1}+(1-\alpha)\nu_{\mathcal O_2}$ uniformly in $U$.
\end{lemma}

We need a few preliminary propositions before we can pass to the proof of Lemma~\ref{averaging_lemma}.

For any $c_0\in\bbC$ and a periodic orbit $\mathcal O$ of $f_{c_0}$ that is non-critical and not primitively parabolic, 
let $U_{\mathcal O}\subset\bbC$ be a simply connected neighborhood of $c_0$, such that $\rho_{\mathcal O}(c)\neq 0$ for any $c\in U_{\mathcal O}$ and let $g_{\mathcal O}\colon U_{\mathcal O}\to\bbC$ be the analytic map defined by the relation
\begin{equation}\label{g_cO_def_eq}
g_{\mathcal O}(c):= (\rho_{\mathcal O}(c))^{1/|\mathcal O|},
\end{equation}
where the branch of the root is chosen so that 
$$
\arg(g_{\mathcal O}(c))\in (-\pi/|\mathcal O|, \pi/|\mathcal O|].
$$
(A particular choice of the branch of the root is not important, but we prefer to make a definite choice.)

For further reference, let us make the following basic observation:
\begin{proposition}\label{log_der_prop}
For any $c_0\in\bbC$, a non-critical periodic orbit $\mathcal O$ of $f_{c_0}$ and a neighborhood $U_{\mathcal O}\subset\bbC$, satisfying the above conditions, we have
$$
\frac{d}{dc}[\log (g_{\mathcal O}(c))] = \nu_{\mathcal O}(c),
$$
for all $c\in U_{\mathcal O}$.
\end{proposition}
\begin{proof}
This follows from a basic computation.
\end{proof}

\begin{proposition}\label{nest_prop}
Assume, $z_0\in\bbC$ is a periodic point that belongs to a repelling periodic orbit $\mathcal O$ of period $n$ for a map $f_{c_0}$, where $c_0\in\bbC$ is an arbitrary fixed parameter. 
Let $V\subset\bbC$ be a simply connected neighborhood of $z_0$, such that $f_{c_0}^{\circ n}$ is univalent on $V$ and for an appropriate branch of the inverse $f_{c_0}^{\circ (- n)}$, the inclusion $f_{c_0}^{\circ (- n)}(V)\Subset V$ holds. 
Then there exists a neighborhood $U\subset\bbC$ of $c_0$, such that for all $c\in U$, the inverse branch $f_{c}^{\circ (- n)}$ is defined on $V$, the inclusion $f_{c}^{\circ (- n)}(V)\Subset V$ holds, and for any $z\in V$, the analytic functions
$$
h_{k,z}(c):= [(f_c^{\circ (nk)})'(f_c^{\circ(-nk)}(z))]^{1/(nk)}
$$
converge to $g_{\mathcal O}$ uniformly in $z\in V$ and $c\in U$, for appropriate branches of the roots, as $k\to\infty$.
\end{proposition}
\begin{proof}
Since the inverse branch $f_{c_0}^{\circ (- n)}$ taking $V$ compactly inside itself, is defined on a domain that compactly contains $V$, it follows that the same holds for $f_{c}^{\circ (- n)}$, where $c$ is any parameter from a sufficiently small neighborhood $U$ of $c_0$.

According to Denjoy-Wolff Theorem, for any $c\in U$, the map $f_c^{\circ n}$ has a unique repelling fixed point $z_c$ that depends analytically on $c$ and coincides with $z_0$, when $c=c_0$. This implies that the map $g_{\mathcal O}$ is defined for all $c\in U$.

Finally, since for any $c\in U$ and $z\in V$, the sequence of points $\{f_c^{\circ(-nk)}(z)\}_{k=1}^\infty$ converges to $z_c$ uniformly in $z\in V$, it follows that
$$
\lim_{k\to\infty}h_{k,z}(c) = ((f_c^{\circ n})'(z_c))^{1/n} = g_{\mathcal O}(c),
$$
assuming that appropriate branches of the roots are chosen in the definition of $h_{k,z}(c)$.
\end{proof}

\begin{proposition}\label{dense_preimage_prop}
Let $c,z_0\in\bbC$ be such that $z_0$ is a repelling periodic point of $f_c$. Assume that $(c, z_0)\neq (-2,2)$. then there exists a sequence $z_{-1}, z_{-2},z_{-3},\ldots\in\bbC$ such that the following holds simultaneously:
\begin{enumerate}[(i)]
	\item\label{dense_prop_cond_1} the sequence $z_{-1}, z_{-2},z_{-3},\ldots$ is dense in the Julia set $J_c$;
	\item\label{dense_prop_cond_2} $f(z_{-j})=z_{1-j}$, for any $j\in\bbN$;
	\item\label{dense_prop_cond_3} $z_{-j}\neq 0$, for any $j\in\bbN$.	
\end{enumerate}
\end{proposition}
\begin{proof}
Existence of a sequence that satisfies~\ref{dense_prop_cond_1} and~\ref{dense_prop_cond_2}, follows immediately from the fact that the set of preimages of any point in the Julia set $J_c$ is dense in $J_c$. Indeed, from any point $z_{-k}$ one can land in any arbitrarily small region of $J_c$, by taking an appropriate sequence of preimages of $z_{-k}$. We can continue this process, making sure that any arbitrarily small region of $J_c$ is eventually visited by our sequence.
Furthermore, property~\ref{dense_prop_cond_2} implies that if $z_{-k}$ does not belong to the periodic orbit of $z_0$, then for every $j\ge k$, the element $z_{-j}$ is different from any other element of the entire sequence $z_0,z_{-1},z_{-2},\ldots$, no matter, how the sequence of preimages of $z_{-k}$ was chosen.

Property~\ref{dense_prop_cond_3} is equivalent to the property that $z_{-j}\neq c$, for any $j\in\bbN\cup\{0\}$, since $c$ is the unique point that has only one preimage under the map $f_c$, and that preimage is $0$. 

Let $\mathcal O$ be the periodic orbit of $f_c$ that contains $z_0$. First of all, we note that $c\not\in\mathcal O$. Otherwise, if $c\in\mathcal O$, then $0\in\mathcal O$, since $0$ is the unique preimage of $c$, and the orbit $\mathcal O$ is super-attracting, which contradicts the assumption of the proposition.

Assume that the sequence, constructed in the first paragraph of the proof, violates property~\ref{dense_prop_cond_3}. Let $j\in\bbN$ be such that $z_{-j}=c$. This number $j$ is unique, since $c\not\in\mathcal O$, so all further preimages of $c$ must differ from $c$.
If $z_{1-j}\not\in\mathcal O$, then we can modify $z_{-j}$ by taking it to be equal to another preimage of $z_{1-j}$. After that we can construct the remaining ``tail'' of the sequence by the same process, as described in the first paragraph. Since $z_{1-j}\not\in\mathcal O$, no further element of the sequence will ever return to $z_{1-j}$, hence, the sequence is guaranteed to avoid the critical value $c$.

It follows from the construction, described in the previous paragraph, that the sequence $z_{-1},z_{-2},\ldots$ satisfying properties~\ref{dense_prop_cond_1}-\ref{dense_prop_cond_3}, can be constructed, if at least one point of the periodic orbit $\mathcal O$ has a preimage under $f_c$ that does not belong to $\mathcal O$ and is not simultaneously equal to~$c$. This condition is always satisfied, unless $z_0$ is a fixed point whose two preimages are $z_0$ and $c$. The latter happens only when $c=-2$ and $z_0=2$.
\end{proof}

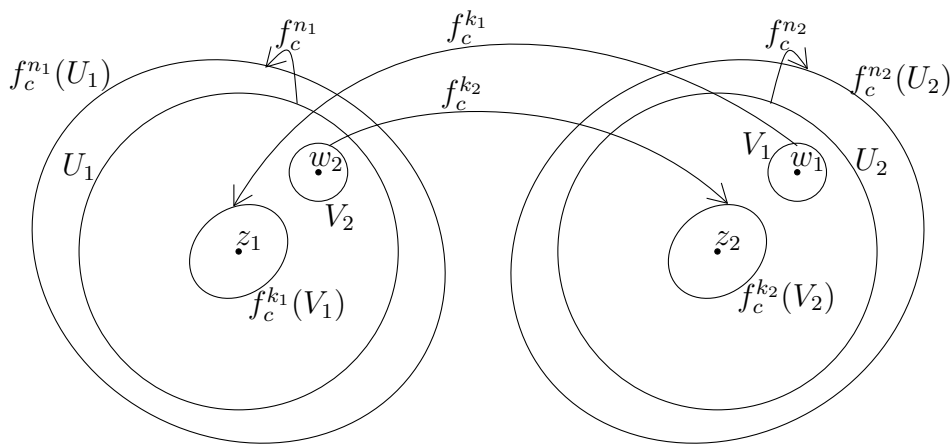
\begin{figure}[h!]
\begin{tikzpicture}[scale =0.7]
\draw (0,0) circle (3);

\filldraw (0,0) circle (0.05);
\draw (0.2,0.2) node {$z_1$};
\filldraw (1.5,1.5) circle (0.05);
\draw (1.65,1.7) node {$w_2$};
\draw (1.9,0.65) node {$V_2$};
\draw (1.5,1.5) circle (0.55);
\draw (-3.35,3.35) node {$f_c^{n_1}(U_1)$};
\draw[-{Straight Barb[angle'=60,scale=3]}] (1.7,2) .. controls (3.3, 3) and (7, 3) .. (9.2,0.9);
\draw (4.3,4.25) node {$f_c^{k_1}$};
\draw (4.2, 3) node{$f_c^{k_2}$};
\draw[-{Straight Barb[angle'=60,scale=3]}] (10.5,2) .. controls (7.3, 4.5) and (2.3, 5) .. (-0.1,0.85);
\draw[-{Straight Barb[angle'=60,scale=3]}] (1.1,2.8) .. controls (1, 4) .. (0.5,3.48);
\draw[-{Straight Barb[angle'=60,scale=3]}] (10,2.8) .. controls (10.3, 4) .. (10.7,3.45);
\draw (1.1, 4.1) node{$f_c^{n_1}$};
\draw (10.3, 4.1) node {$f_c^{n_2}$};

\draw[rotate=60] (0,0) ellipse (3.5 and 4.0);
\draw (-3,1.6) node {$U_1$};

\draw[rotate around={40:(0,0)}] (0,0) ellipse (1 and 0.8);
\draw (1.1,-1) node {$f_c^{k_1}(V_1)$};

\draw (9,0) circle (3);

\filldraw (9,0) circle (0.05);
\draw (9.2,0.2) node {$z_2$};
\filldraw (10.5,1.5) circle (0.05);
\draw (10.7,1.7) node {$w_1$};
\draw (9.75,2) node {$V_1$};
\draw (10.5,1.5) circle (0.55);
\draw (12.45,3.25) node {$f_c^{n_2}(U_2)$};

\draw[rotate around ={-60:(9,0)}] (9,0) ellipse (3.5 and 4.0);
\draw (11.9,1.7) node {$U_2$};

\draw[rotate around={40:(9,0)}] (9,0) ellipse (1 and 0.8);
\draw (10.3,-0.85) node {$f_c^{k_2}(V_2)$};
\end{tikzpicture}
\caption{Maps and domains from the proof of Lemma~\ref{averaging_lemma}.}\label{Av_lemma_pic}
\end{figure}

\begin{proof}[Proof of Lemma~\ref{averaging_lemma}]
Let $n_1$ and $n_2$ be the periods of the periodic orbits $\mathcal O_1$ and $\mathcal O_2$ respectively. Let $z_1$ and $z_2$ be some periodic points from each of the orbits $\mathcal O_1$ and $\mathcal O_2$. Since the orbits $\mathcal O_1$ and $\mathcal O_2$ are repelling, there exist a simply connected neighborhood $U$ of $c_0$ and two neighborhoods $U_1$ and $U_2$ of $z_1$ and $z_2$ respectively, such that for all $c\in U$, the maps $f_c^{\circ n_1}$ and $f_c^{\circ n_2}$ are univalent on $U_1$ and $U_2$ respectively, and $f_c^{\circ n_1}(U_1)\setminus U_1$ and $f_c^{\circ n_2}(U_2)\setminus U_2$ are two annuli.

According to Proposition~\ref{dense_preimage_prop}, there exist $k_1,k_2\in\bbN$, $w_1\in U_2$ and $w_2\in U_1$, such that 
$$
f_{c_0}^{\circ k_1}(w_1)=z_1, \qquad f_{c_0}^{\circ k_2}(w_2)=z_2,
$$ 
$$
(f_{c_0}^{\circ k_1})'(w_1)\neq 0, \quad\text{and}\quad (f_{c_0}^{\circ k_2})'(w_2)\neq 0.
$$ 
Possibly, after shrinking the neighborhood $U$ of $c_0$, there exist a constant $K>1$ and the neighborhoods $V_1\Subset U_2$ and $V_2\Subset U_1$ of $w_1$ and $w_2$ respectively, such that for any $c\in U$ and $j\in\{1,2\}$, the following holds (see Figure~\ref{Av_lemma_pic}):

\begin{enumerate}[(a)]
	\item $f_c^{\circ k_j}$ is univalent on $V_j$ and maps it inside $U_j$.
	\item\label{V_j_property_2} The neighborhood $f_c^{\circ k_j}(V_j)$ contains a repelling periodic point of period $n_j$ for the map $f_c$. (For $c=c_0$, this periodic point is $z_j$, while for other $c\in U$ it is its perturbation.)
	\item For any $z\in V_j$, we have 
	\begin{equation}\label{K_inequalities}
	K^{-k_j}<|(f_c^{\circ k_j})'(z)|<K^{k_j}.	
	\end{equation}
\end{enumerate}

Let $N\in\bbN$ be a sufficiently large number, such that for any $N_1, N_2\ge N$ and any $c\in U$, we have
\begin{equation}\label{V_j_inclusions}
f_c^{\circ (-n_1N_1)}(V_2) \Subset  f_c^{\circ k_1}(V_1)\quad\text{and}\quad 
f_c^{\circ (-n_2N_2)}(V_1) \Subset  f_c^{\circ k_2}(V_2),
\end{equation}
for the inverse branches of $f_c$ that take $U_1$ into itself in the case of the first inclusion, and $U_2$ into itself in the case of the second one. Existence of such a number $N$ follows from property~\ref{V_j_property_2}. 

Assume, $N_1,N_2\in\bbN$ satisfy the condition $N_1,N_2\ge N$. Then for every $c\in U$, one may consider the following composition of inverse branches of $f_c$:
$$
V_1\xrightarrow{\enspace f_c^{\circ (-n_2N_2)} \enspace} f_c^{\circ k_2}(V_2) \xrightarrow{\enspace f_c^{\circ (-k_2)} \enspace}
V_2 \xrightarrow{\enspace f_c^{\circ (-n_1N_1)} \enspace} f_c^{\circ k_1}(V_1) \xrightarrow{\enspace f_c^{\circ (-k_1)} \enspace} V_1.
$$
Let us denote this composition by $h_c\colon V_1\to V_1$. By construction, this is a univalent map, and the inclusions~(\ref{V_j_inclusions}) imply that $h_c(V_1)\Subset V_1$. Then, according to the Denjoy-Wolff Theorem, the map $h_c$ has a unique fixed point in $V_1$, which is a repelling periodic point of period
$$
M=n_1N_1+n_2N_2+k_1+k_2
$$
for the map $f_c$. Let $\mathcal O_{N_1,N_2}$ denote the periodic orbit of such a point when $c=c_0$. Then the map $g_{\mathcal O_{N_1,N_2}}$ is defined in $U$. 

After, possibly shrinking the neighborhood $U$ of $c_0$, we may apply Proposition~\ref{nest_prop} for $V=V_1$ and $V=V_2$. A direct computation shows that for appropriate branches of the roots, we have
\begin{equation}\label{g_c_OO_eq}
g_{\mathcal O_{N_1,N_2}}(c) = (h_{N_1,\hat z}(c))^{\frac{n_1N_1}{M}} (h_{N_2,\tilde z}(c))^{\frac{n_2N_2}{M}}
(\beta(c))^{\frac{k_1}{M}}(\gamma(c))^{\frac{k_2}{M}},
\end{equation}
where $h_{N_1,\hat z}$ and $h_{N_2,\tilde z}$ are the functions from Proposition~\ref{nest_prop}, $\hat z\in V_2$ and $\tilde z\in V_1$ are some points that depend on $N_1$, $N_2$ and $c\in U$, and the holomorphic functions $\beta$ and $\gamma$ satisfy
$$
K^{-1}<|\beta(c)|, |\gamma(c)|<K,
$$
where $K$ is the same as in~(\ref{K_inequalities}). 
Now, Proposition~\ref{nest_prop} and~(\ref{g_c_OO_eq}) imply that if $N_1,N_2\to\infty$, so that
$$
\frac{n_1N_1}{n_1N_1+n_2N_2}\to\alpha,
$$
then
\begin{equation}\label{g_c_OO_convergence_eq}
g_{\mathcal O_{N_1,N_2}}(c) \to s\cdot (g_{\mathcal O_1}(c))^\alpha (g_{\mathcal O_2}(c))^{1-\alpha},
\end{equation}
uniformly in $c\in U$, for appropriate fixed branches of the degree maps $z\mapsto z^\alpha$ and $z\mapsto z^{1-\alpha}$, and some constant $s\in\bbC$, such that $|s|=1$.

Finally, the proof of Lemma~\ref{averaging_lemma} can be completed by taking logarithmic derivatives of both sides in~(\ref{g_c_OO_convergence_eq}) and applying Proposition~\ref{log_der_prop}.
\end{proof}

\section{The sets $\mathcal Y_c$}
We start this section by giving a proof of Theorem~\ref{main_theorem_2}. 
We note that our proof of part~\ref{YX_dichotomy} of Theorem~\ref{main_theorem_2}, providing the necessary and sufficient condition for $c\in\bbC\setminus\M$ to be contained in $\mathcal X$, seriously depends on the assumption that $c\not\in\M$. Furthermore, the condition itself seems to be wrong for some $c\in\partial\M$, (c.f. Remark~\ref{where_proof_rails_remark}). Indeed, the case $c\in\M$ appears to be more delicate. In the second part of this section we provide a sufficient condition for $c\in\M$ to be contained in $\mathcal X$. Later, in Section~\ref{M_in_X_section} we show that this condition is satisfied for any $c\in\M$.

\subsection{Proof of Theorem~\ref{main_theorem_2}}
In order to prove property~\ref{YX_dichotomy} of Theorem~\ref{main_theorem_2}, we need the following lemma:

\begin{lemma}\label{nu_normality_lemma}
	For any $c\in\bbC\setminus\partial\M$, the family of maps $\{\nu_{\mathcal O}\mid \mathcal O\in\Omega_c\}$ is defined and is normal on any simply connected neighborhood $U\subset\bbC$, such that $c\in U$ and $U\cap\partial\M=\varnothing$. Furthermore, if $c\in\bbC\setminus\M$, then the identical zero is not a limiting map of the normal family $\{\nu_{\mathcal O}\mid \mathcal O\in\Omega_c\}$.
\end{lemma}
\begin{proof}
	Fix $c\in\bbC\setminus\partial\M$ and a neighborhood $U$ as in the statement of the lemma. Since $U\cap\partial\M=\varnothing$, all repelling periodic orbits of $f_{c_0}$ remain to be repelling after analytic continuation in $c\in U$. This implies that all maps from the family 
	$$
	\mathcal G_c:=\{g_{\mathcal O}\mid \mathcal O\in\Omega_c\},
	$$
	are defined in the neighborhood $U$ and are analytic in it. (We recall that the maps $g_{\mathcal O}$ were defined in~(\ref{g_cO_def_eq}) and are appropriate branches of the roots of the multipliers.) Furthermore, since all functions from $\mathcal G_c$ are locally uniformly bounded, 
	the family $\mathcal G_c$ is normal in $U$. Together with Proposition~\ref{log_der_prop}, this implies normality of the family $\{\nu_{\mathcal O}\mid \mathcal O\in\Omega_c\}$.

	If $c\in\bbC\setminus\M$, then without loss of generality we may assume that the domain $U$ is simply connected and unbounded.  
	Since for all $\tl c\in\bbC$ sufficiently close to $\infty$, the Julia set $J_{\tl c}$ is contained in the annulus centered at zero with inner and outer radii being equal to $\sqrt{|\tl c|}\pm 1$, it follows that for every $\tl c\in U$ sufficiently close to $\infty$ and for any $\mathcal O\in\Omega_c$, we have
	\begin{equation}\label{g_O_double_inequality}
	2\sqrt{|\tl c|}-2<|g_{\mathcal O}(\tl c)|< 2\sqrt{|\tl c|}+2,
	\end{equation}
	which implies that none of the limiting maps of the family $\mathcal G_c$ is a constant map. Then it follows that the identical zero is not a limiting map of the normal family $\{\nu_{\mathcal O}\mid \mathcal O\in\Omega_c\}$.
\end{proof}

\begin{proof}[Proof of Theorem~\ref{main_theorem_2}]
First, we observe that property~\ref{Y_convex_property} of Theorem~\ref{main_theorem_2} is an immediate corollary from the Averaging Lemma (Lemma~\ref{averaging_lemma}). Indeed, if $c\neq-2$, then convexity of $\mathcal Y_c$ is obvious from Lemma~\ref{averaging_lemma}. On the other hand, if $c=-2$, then according to the same lemma, the set $\mathcal Y_{-2}$ is the union of a convex set and a single point $\nu_{\{2\}}(-2)$, corresponding to the periodic orbit $\mathcal O=\{2\}$. A direct computation shows that 
$$
\rho_{\{2\}}(-2)= 4, \qquad \rho_{\{2\}}'(-2)= -2/3,
$$
hence, $\nu_{\{2\}}(-2)=-1/6$.

We proceed with the proof of part~\ref{YX_dichotomy} as follows: for $c\in\bbC\setminus\M$, let $U$ be a neighborhood of $c$ that satisfies the conditions of Lemma~\ref{nu_normality_lemma}. First, we observe that according to Lemma~\ref{nu_normality_lemma}, the family $\{\nu_{\mathcal O}\mid \mathcal O\in\Omega_c\}$, defined on $U$, is locally uniformly bounded, hence, the set $\mathcal Y_c$ is bounded.

\underline{\textit{Necessary condition for $c\in\mathcal X$}}: If $c\in\mathcal X$, then 
there exists a sequence of points $\{c_k\}_{k=1}^\infty$ and a sequence of periodic orbits $\{\mathcal O_k\}_{k=1}^\infty\subset\Omega_c$, such that 
$$
\lim_{k\to\infty} c_k =c\qquad\text{and}\qquad \rho_{\mathcal O_k}'(c_k)=0,\quad\text{for any } k\in\bbN.
$$
According to Lemma~\ref{nu_normality_lemma}, after extracting a subsequence, we may assume that the sequence of maps $\nu_{\mathcal O_k}$ converges to some holomorphic map $\nu\colon U\to\bbC$ uniformly on compact subsets of $U$. Since for any $k\in\bbN$, we have $\nu_{\mathcal O_k}(c_k)=0$, it follows by continuity that $\nu(c)=0$. Finally, convergence of the maps $\nu_{\mathcal O_k}$ to $\nu$ implies that
$$
\lim_{k\to\infty}\nu_{\mathcal O_k}(c)=\nu(c)=0,
$$
hence, $0\in\mathcal Y_c$.

\underline{\textit{Sufficient condition for $c\in\mathcal X$}}: On the other hand, if $0\in\mathcal Y_c$, then either there exists a periodic orbit $\mathcal O\in\Omega_c$, such that $\nu_{\mathcal O}(c)=0$ or there exists a sequence of periodic orbits $\{\mathcal O_k\}_{k=1}^\infty\subset\Omega_c$, such that 
$$
\lim_{k\to\infty} \nu_{\mathcal O_k}(c)=0.
$$
In the first case, $\rho_{\mathcal O}'(c)=0$, so $c\in\mathcal X$ according to Theorem~\ref{XoutsideM_theorem}.

In the second case, according to Lemma~\ref{nu_normality_lemma}, after extracting a subsequence, we may assume that the sequence of maps $\nu_{\mathcal O_k}$ converges to some holomorphic map $\nu\colon U\to\bbC$ uniformly on compact subsets of $U$. By continuity, we have $\nu(c)=0$, and, according to Lemma~\ref{nu_normality_lemma}, $\nu\not\equiv 0$. Then it follows from Rouch\'e's Theorem that for any sufficiently large $k\in\bbN$, there exists $c_k\in U$, such that $\nu_{\mathcal O_k}(c_k)=0$ and $\lim_{k\to\infty}c_k=c$. The latter implies that $c\in\mathcal X$, and completes the proof of Theorem~\ref{main_theorem_2}. 
\end{proof}

\begin{remark}\label{where_proof_rails_remark}
The above proof of part~\ref{YX_dichotomy} of Theorem~\ref{main_theorem_2} fails without the assumption $c\not\in\M$. Indeed, if $c\in\partial\M$, then the neighborhood $U$ from Lemma~\ref{nu_normality_lemma} does not exist. Furthermore, even though $\partial\M\subset\mathcal X$ (since $\partial\M$ is the support of the bifurcation measure $\mbif$) and $-2\in\partial\M$, the preliminary computations indicate that the set $\mathcal Y_{-2}$ seems to be disjoint from~$0$. In the case $c\in\M\setminus\partial\M$, the above proof of the sufficient condition for $c\in\mathcal X$ fails, since the limiting map $\nu$ might turn out to be the identical zero. 
\end{remark}

\subsection{A sufficient condition for $c\in\M$ to be contained in $\mathcal X$}
In this subsection we prove the following sufficient condition for $c\in\bbC\setminus\partial\M$ to be contained in $\mathcal X$.  

\begin{lemma}\label{suff_cond_lemma}
Let $c\in\bbC\setminus\partial\M$ be an arbitrary parameter. If there exist finitely many repelling periodic orbits $\mathcal O_1,\mathcal O_2,\ldots,\mathcal O_k\in\Omega_c$, such that $0$ is contained in the convex hull of the points $\nu_{\mathcal O_1}(c),\ldots, \nu_{\mathcal O_k}(c)$, then $c\in\mathcal X$.
\end{lemma}

When $c\in\bbC\setminus\M$, the sufficient condition, given by Lemma~\ref{suff_cond_lemma}, is an immediate corollary of Theorem~\ref{main_theorem_2}, but we will use Lemma~\ref{suff_cond_lemma} for $c\in\M\setminus\partial\M$.

First, in order to prove Lemma~\ref{suff_cond_lemma}, we need the following proposition:

\begin{proposition}\label{not_id_zero_prop}
Let $c\in\bbC$ be an arbitrary parameter and let $\mathcal O_1,\mathcal O_2,\ldots,\mathcal O_k\in\Omega_c$ be a finite collection of repelling periodic orbits. If $\alpha_1,\ldots,\alpha_k\in\bbR$ are such that $\sum_{j=1}^k\alpha_j\neq 0$, then the map 
$$
\nu:= \sum_{j=1}^k\alpha_j\nu_{\mathcal O_j},
$$
defined in a neighborhood of the point $c$, is not a constant map.
\end{proposition}
\begin{proof}
Since for every $j=1,\ldots,k$, the multipliers $\rho_{\mathcal O_j}$ are algebraic (multiple-valued) maps, it follows from~(\ref{def_nu_eq}) that the map $\nu$ has a single-valued meromorphic extension to any simply-connected domain $U\subset\bbC$ that avoids finitely many branching points of the maps $\rho_{\mathcal O_j}$. 
Note that none of the branching points lie on the real ray $(-\infty,-3)$, since $(-\infty,-3)\cap\M=\varnothing$. Furthermore, since for any parameter $\tl c\in (-\infty,-3)$, the corresponding Julia set $J_{\tl c}$ lies on the real line, it follows that all maps $\rho_{\mathcal O_j}$ take real values when restricted to the ray $(-\infty,-3)$. 
Choose the domain $U$ so that it is unbounded and $(-\infty,-3)\subset U$. 
Then for any $j=1,\ldots,k$, we have the same asymptotic relation
$$
\rho_{\mathcal O_j}(\tl c)\sim \pm(-4\tl c)^{|\mathcal O_j|/2},
$$
as $\tl c\to -\infty$ within the domain $U$. A direct computation yields that $\nu_{\mathcal O_j}(\tl c)\sim 1/(2\tl c)$, hence, 
$$
\nu(\tl c)\sim \frac{\sum_{j=1}^k\alpha_j}{2\tl c},
$$ 
as $\tl c\to -\infty$ within the domain $U$. Since $\sum_{j=1}^k\alpha_j\neq 0$, the latter implies that $\nu$ is not a constant map.
\end{proof}

\begin{proof}[Proof of Lemma~\ref{suff_cond_lemma}]
Since the convex hull of the points $\nu_{\mathcal O_1}(c),\ldots, \nu_{\mathcal O_k}(c)$ contains zero, it follows that there exist real non-negative constants $\alpha_1,\ldots,\alpha_k$, such that $\sum_{j=1}^k\alpha_j=1$ and the analytic map
$$
\nu:= \sum_{j=1}^k\alpha_j\nu_{\mathcal O_j},
$$
defined in some neighborhood of the point $c$, satisfies $\nu(c)=0$.

Since $c\not\in\partial\M$, this means that $c\neq -2$, so it follows from the Averaging Lemma (Lemma~\ref{averaging_lemma}) that there exists a sequence of periodic orbits $\{\mathcal O_m'\}_{m=1}^\infty\subset\Omega_c$ and a neighborhood $U\subset\bbC$ of the point $c$, such that all maps $\nu_{\mathcal O_m'}$ are defined and analytic in $U$ and
$$
\nu_{\mathcal O_m'}\to\nu\qquad\text{as }m\to\infty,\qquad\text{uniformly on $U$}.
$$
According to Proposition~\ref{not_id_zero_prop}, the map $\nu$ is not the identical zero map.
Now, since $\nu(c)=0$, it follows from Rouch\'e's Theorem that for any sufficiently large $m\in\bbN$, the map $\nu_{\mathcal O_m'}$ has a zero at some point $c_m\in U$, and the points $c_m$ can be chosen so that $\lim_{m\to\infty}c_m=c$. The latter implies that $c\in\mathcal X$.
\end{proof}

\section{Proof of Theorem~\ref{main_theorem_1}}

In this section we complete the proof of Theorem~\ref{main_theorem_1}. 

\subsection{The set $\mathcal X$ is bounded}

First, we prove the following:

\begin{lemma}\label{bounded_X_lemma}
The set $\mathcal X$ is bounded.
\end{lemma}
\begin{proof}
For a fixed parameter $c_0\in\bbC\setminus\M$, the Julia set $J_{c_0}$ of the map $f_{c_0}$ is a Cantor set, and all periodic orbits of $f_{c_0}$ are repelling. 
For any periodic orbit $\mathcal O$ of $f_{c_0}$, the locally defined map $g_{\mathcal O}$ can be extended by analytic continuation to an analytic map of a double cover of the complement of the Mandelbrot set $\M$ (see~\cite{Firsova_Gor_equi} for details). This means that if
$$
\phi_\M\colon\bbC\setminus\M\to\bbC\setminus\overline{\bbD}
$$
is a fixed conformal diffeomorphism of $\bbC\setminus\M$ onto $\bbC\setminus\overline{\bbD}$ and $\lambda_0\in\bbC\setminus\overline{\bbD}$ is a fixed point, such that $\phi_\M^{-1}(\lambda_0^2)=c_0$, then the map 
$$
\lambda\mapsto g_{\mathcal O}(\phi_\M^{-1}(\lambda^2)),
$$
defined for all $\lambda$ in a neighborhood of $\lambda_0$, extends to a global holomorphic map
$$
\gamma_\mathcal O\colon \bbC\setminus\overline{\bbD}\to \bbC\setminus\overline{\bbD}.
$$

Now assume that the statement of Lemma~\ref{bounded_X_lemma} does not hold. Then there exists a sequence of parameters $\{\lambda_n\}_{n\in\bbN}$ and a corresponding sequence of periodic orbits $\{\mathcal O_n\}_{n\in\bbN}$, such that
\begin{equation}\label{seq_eq}
\lim_{n\to\infty}\lambda_n=\infty \qquad\text{and}\qquad \gamma_{\mathcal O_n}'(\lambda_n)=0, \quad\text{for every }n\in\bbN.
\end{equation}
Since the family of maps $\{\gamma_{\mathcal O}\}$ is locally uniformly bounded, hence, normal (c.f. Proposition~5.8 from~\cite{Firsova_Gor_equi}), it follows that after extracting a subsequence, we may assume that the sequence of maps $\gamma_{\mathcal O_n}$ converges to a holomorphic map $\gamma\colon \bbC\setminus\overline{\bbD}\to \bbC\setminus\overline{\bbD}$ uniformly on compact subsets. Since for any $\tl c\in\bbC$ sufficiently close to $\infty$ and any $\mathcal O\in\Omega_{c_0}$, inequality~(\ref{g_O_double_inequality}) holds, we conclude that $\gamma$, as well as each $\gamma_{\mathcal O_n}$, are  non-constant maps that have a simple pole at infinity. On the other hand,~(\ref{seq_eq}) implies that $\gamma$ has at least a double pole at infinity, which provides a contradiction.
\end{proof}

Next, we proceed with proving the remaining statements of Theorem~\ref{main_theorem_1}. 

\subsection{The set $\mathcal X\setminus\M$}
First we study the set $\mathcal X\setminus\M$, i.e., the portion of the set $\mathcal X$ that is contained in the complement of the Mandelbrot set. We note that even though numerical computations from~\cite{Belova_Gorbovickis} together with Theorem~\ref{XoutsideM_theorem}, suggest that this set is non-empty, a rigorous computer-free proof of this fact has not been provided so far. We fill this gap by proving the following:

\begin{lemma}\label{X_minus_M_nonempty_lemma}
The set $\mathcal X\setminus\M$ has non-empty interior.
\end{lemma}

The idea of the proof of Lemma~\ref{X_minus_M_nonempty_lemma} is to show that the sufficient condition from Lemma~\ref{suff_cond_lemma} is satisfied for all $c$ in a neighborhood of the parabolic parameter $c_0=-3/4$. The rest of the proof is technical. We will need explicit formulas for the maps $\nu_{\mathcal O}$, corresponding to periodic orbits $\mathcal O$ of periods $1$, $2$ and $3$.

\begin{proposition}\label{nu_comp_prop}
Let $c_0\in\bbC$ and a corresponding periodic orbit $\mathcal O$ of $f_{c_0}$ be such that the map $\nu:= \nu_{\mathcal O}$ is defined in a neighborhood of the point $c=c_0$. Then the following holds:
\begin{enumerate}[(i)]
\item\label{nu_formula_item1} If $|\mathcal O|=1$, then 
$$
\nu(c) = \frac{2}{4c-1-\sqrt{1-4c}},
$$
where the two branches of the root correspond to the two different periodic orbits of period $1$.
\item\label{nu_formula_item2} If $|\mathcal O|=2$, then 
$$
\nu(c) = \frac{1}{2c+2}.
$$
\item\label{nu_formula_item3} If $|\mathcal O|=3$, then 
$$
\nu(c) = \frac{12c^3+37c^2+32c+7 -(c^2+6c+7)\sqrt{-4c-7}}{6(4c+7)(c^3+2c^2+c+1)},
$$
where the two branches of the root correspond to the two different periodic orbits of period $3$.
\end{enumerate}
\end{proposition}
\begin{proof}
When $|\mathcal O|=1$, i.e, $\mathcal O$ is a fixed point $z$, solving the equation $f_c(z)=z$ yields
$$
\rho_{c_0,\mathcal O}(c) = 2z=1+\sqrt{1-4c}.
$$
Then after a direct computation we get
$$
\nu(c) = \frac{\rho_{c_0,\mathcal O}'(c)}{\rho_{c_0,\mathcal O}(c)} = \frac{2}{4c-1-\sqrt{1-4c}}.
$$

When $|\mathcal O|=2$, there is only one periodic orbit of period~$2$. Its multiplier is the free term of the polynomial
$$
p(z)=\frac{4(f_c^{\circ 2}(z)-z)}{f_c(z)-z}=4z^2+4z+4(c+1).
$$
Now, a direct computation yields the formula for $\nu(c)$ in part~\ref{nu_formula_item2} of the proposition.

Finally, in the case $|\mathcal O|=3$, there are two periodic orbits of period~$3$ and according to~\cite{Stephenson_91}, the multiplier $\rho=\rho(c)$ of each of these orbits satisfies the equation
$$
c^3+2c^2+(1-\rho/8)c+(1-\rho/8)^2=0.
$$
After solving this equation for $\rho$, we obtain
$$
\rho(c) = 8+4c-4c\sqrt{-4c-7}.
$$
Then a direct computation yields the formula for $\nu(c)$ in part~\ref{nu_formula_item3} of the proposition.
\end{proof}

\begin{proof}[Proof of Lemma~\ref{X_minus_M_nonempty_lemma}]
We consider the maps $\nu_{\mathcal O}$ in a neighborhood of the point $c=-3/4$ for periodic orbits $\mathcal O$ of periods $1$, $2$ and $3$. The parameter $c=-3/4$ is the point at which the hyperbolic component of period~$2$ touches the main cardioid of the Mandelbrot set. In particular, all considered functions are defined and analytic in a neighborhood $U$ of that point.

For each $c\in U$, let $H_c$ denote the convex hull of the finite set $\{\nu_{\mathcal O}(c)\mid |\mathcal O|=1, 2, 3\}$. It follows from Proposition~\ref{nu_comp_prop} that $\nu_{\mathcal O}(-3/4)$ is equal to
\begin{itemize}
	\item $-1$ or $-1/3$, when $|\mathcal O|=1$,
	\item $2$, when $|\mathcal O|=2$,
	\item $-\cfrac{10}{183}\pm\cfrac{49}{183}i$, when $|\mathcal O|=3$,
\end{itemize}
hence, $H_{-3/4}$ contains $0$ in its interior. By continuity, it follows that the convex hull $H_c$ contains $0$, for all $c$ in some open complex neighborhood $V$ of the point $-3/4$. Since $c=-3/4$ is a parabolic parameter, it follows that $V\setminus\M$ is a nonempty open set. According to Lemma~\ref{suff_cond_lemma}, we observe that $V\setminus\M\subset\mathcal X$, which completes the proof of Lemma~\ref{X_minus_M_nonempty_lemma}.
\end{proof}

Next, we prove the following:

\begin{lemma}\label{connectedness_lemma}
The set $\mathcal X\cup\M$ is path connected.
\end{lemma}
\begin{proof}
First, let us note that for any $c_0\in\bbC\setminus\M$, any periodic orbit $\mathcal O_0\in\Omega_{c_0}$ of $f_{c_0}$ and any piecewise smooth curve $\gamma\colon[0,1]\to\bbC\setminus\M$, such that $\gamma(0)=c_0$, the periodic orbit $\mathcal O_0$ can be analytically continued along the curve $\gamma$. Since all periodic orbits of $f_c$ are repelling, when $c\in\bbC\setminus\M$, this defines analytic continuation of the locally defined map $\nu_{\mathcal O_0}$ along the curve $\gamma$. In particular, this means that if $\nu$ is an analytic map defined in a neighborhood of the point $c_1:=\gamma(1)$ by analytic continuation of $\nu_{\mathcal O_0}$ along $\gamma$, then there exists a periodic orbit $\mathcal O_1\in\Omega_{c_1}$ of $f_{c_1}$, such that $\nu\equiv \nu_{\mathcal O_1}$ in a neighborhood of $c_1$.

Now, according to Lemma~\ref{bounded_X_lemma}, the set $\mathcal X$ is bounded, so there exists an open disk $D\subset\bbC$, such that $\mathcal X\subset D$. 
Let $c_0\in \bbC\setminus\M$ be an arbitrary point for which there exists a periodic orbit $\mathcal O$ of the map $f_{c_0}$, such that $\nu_{\mathcal O}(c_0)=0$. 
Let $\mathcal O_2$ be the unique periodic orbit of period $2$ for the map $f_{c_0}$. Then for each $t\in[0,1]$, we consider the map
$$
\nu_t:=(1-t)\nu_{\mathcal O}+ t\nu_{\mathcal O_2},
$$
defined in a neighborhood of $c_0$. 

Let $S\subset (D\setminus\M)\times[0,1]$ be the set of all points $(c,t)\in (D\setminus\M)\times[0,1]$, such that for some analytic continuation $\tl\nu_t$ of the map $\nu_t$ to a neighborhood of the point $c$, we have $\tl\nu_t(c)=0$. Since there are finitely many different analytic continuations of $\nu_t$ to a fixed neighborhood of a point $c$, it follows that the set $S$ is closed in $(D\setminus\M)\times[0,1]$.
It also follows from the definition of the set $S$ that $(c_0,0)\in S$. Let $S_0\subset S$ be the (path)-connected component of $S$ that contains the point $(c_0,0)$. 
Since according to Proposition~\ref{not_id_zero_prop}, neither the map $\nu_t$, nor any of its analytic continuations is a constant map, this implies that 
the projection of $S_0$ onto the second coordinate is an interval $I\subset[0,1]$ that is open in $[0,1]$.
At the same time, since according to Proposition~\ref{nu_comp_prop}, the map $\nu_1=\nu_{\mathcal O_2}$ does not vanish at any point of the complex plane, it follows that $S_0\cap [(D\setminus\M)\times\{1\}]=\varnothing$, which implies that $1\not\in I$. Since interval $I$ is open in $[0,1]$ and $S_0$ is closed in $(D\setminus\M)\times[0,1]$, we conclude that the closure $\overline S_0$ of $S_0$ in $\bbC\times[0,1]$ has a nonempty intersection with the boundary $(\partial(D\setminus\M))\times[0,1]$.

Finally, note that the Averaging Lemma (Lemma~\ref{averaging_lemma}) together with part~\ref{YX_dichotomy} of Theorem~\ref{main_theorem_2}, imply that the projection of $\overline S_0$ to the first coordinate is contained in some path connected component $X$ of the set $\mathcal X\cup\M$. The latter implies that $X\cap \partial (D\setminus\M)\neq\varnothing$. Since $\mathcal X\cap\partial D=\varnothing$, we conclude that $X\cap\partial\M\neq\varnothing$, so $\M\subset X$. Since $c_0\in\bbC\setminus\M$ was an arbitrary critical point of the multiplier of an arbitrary periodic orbit, and $c_0\in X$, it follows that the path connected set $X$ is dense in $\mathcal X\cup\M$, hence, $X=\mathcal X\cup\M$, and the set $\mathcal X\cup\M$ is path connected as well.
\end{proof}

\subsection{The set $\mathcal X\cap\M$}\label{M_in_X_section}
Here we turn to the study of the portion of the set $\mathcal X$ that is contained in the Mandelbrot set. We show that the whole Mandelbrot set is contained in $\mathcal X$.

\begin{lemma}\label{Inside_M_lemma}
The inclusion $\M\subset\mathcal X$ holds.
\end{lemma}

Before proving Lemma~\ref{Inside_M_lemma}, we need several additional results.

For any $c\in\bbC$ and any $k\in\bbN$, let $\Omega_c^k$ be the set of all periodic orbits of period $k$ for the map $f_c$. (In particular, $\Omega_c^k$ may contain a non-repelling orbit, if it exists.)

\begin{lemma}\label{summation_lemma}
Let $c_0\in\bbC$ be an arbitrary parameter that is neither parabolic, nor critically periodic. Then for any $k\in\bbN$, and the corresponding function $F_k(c):= f_c^{\circ(k-1)}(c)$, the following holds:
\begin{equation}\label{g_k_c_formula}
\frac{F_k'(c_0)}{kF_k(c_0)} = \sum_{m\in\bbN, m|k}\,\,\sum_{\mathcal O\in\Omega_{c_0}^m}\frac{m}{k}\nu_{\mathcal O}(c_0),
\end{equation}
where the summation goes over all $m\in\bbN$, such that $m$ divides $k$ and over all periodic orbits $\mathcal O\in\Omega_{c_0}^m$.
\end{lemma}
\begin{proof}
For every $k\in\bbN$, it follows from Vieta's formulas that $F_k(c_0)$ is the product of all fixed points of the map $f_{c_0}^{\circ k}$, counted with multiplicities. Since $c_0$ is a non-parabolic parameter, all of these fixed points have multiplicity one, hence we have
\begin{equation}\label{g_k_c_auxiliary_formula}
F_k(c_0) = 2^{-2^{k}}\prod_{m\in\bbN, m|k}\,\,\prod_{\mathcal O\in\Omega_{c_0}^m}\rho_{\mathcal O}(c_0).
\end{equation}
Since the parameter $c_0$ is not critically periodic, we have $F_k(c_0)\neq 0$, and for any periodic orbit $\mathcal O$ of $f_{c_0}$, the map $\nu_{\mathcal O}$ is defined and analytic in some fixed neighborhood of the point $c_0$. This implies that both the left hand side and the right hand side of~(\ref{g_k_c_formula}) are defined. Finally, the identity~(\ref{g_k_c_formula}) can be obtained from~(\ref{g_k_c_auxiliary_formula}) by a direct computation.
\end{proof}

Next, we prove a slightly refined version of the Averaging Lemma:

\begin{proposition}\label{averaging_extension_prop}
Under the conditions of Lemma~\ref{averaging_lemma}, if the periods of the periodic orbits $\mathcal O_1$ and $\mathcal O_2$ are relatively prime, then the sequence of repelling periodic orbits $\{\mathcal O_j\}_{j=3}^\infty$ from Lemma~\ref{averaging_lemma} can be chosen so that $|\mathcal O_j|=j$, for any $j\ge 3$.
\end{proposition}
\begin{proof}
Here we refer to the proof of the Averaging Lemma (Lemma~\ref{averaging_lemma}). Define $n_1:=|\mathcal O_1|$ and $n_2:= |\mathcal O_2|$. It was shown that there exist constants $k_1, k_2\in\bbN$ (that depend on $c_0$, $\mathcal O_1$ and $\mathcal O_2$), such that the sequence of orbits $\{\mathcal O_j\}_{j=3}^\infty$ can be chosen to satisfy the following:
$$
|\mathcal O_j|=n_1N_{1,j}+n_2N_{2,j}+k_1+k_2,
$$
for some $N_{1,j},N_{2,j}\in\bbN$, where 
\begin{equation}\label{N_j_conditions}
N_{1,j},N_{2,j}\to\infty\quad\text{and}\quad\frac{n_1N_{1,j}}{n_1N_{1,j}+n_2N_{2,j}}\to\alpha,\qquad\text{as }j\to\infty.
\end{equation}
In order to prove the proposition, it is sufficient to show that for every $\alpha\in[0,1]$, there exist two sequences $\{N_{1,j}\}_{j=3}^\infty$, $\{N_{2,j}\}_{j=3}^\infty$ of positive integers that satisfy~(\ref{N_j_conditions}) and such that 
$$
j=n_1N_{1,j}+n_2N_{2,j}+k_1+k_2,
$$
for all sufficiently large $j\in\bbN$.

It follows from elementary number theory that for every sufficiently large $j\in\bbN$, the  Diophantine equation
\begin{equation}\label{N12_conditions}
j=n_1N_{1}+n_2N_{2}+k_1+k_2,
\end{equation}
has a solution $(N_1, N_2)=(K_1, K_2)\in\bbN^2$ in positive integers. 
Furthermore, the set of all pairs $(N_1,N_2)\in\bbN^2$, satisfying~(\ref{N12_conditions}), can be described as
$$
\mathcal N_j = \{(K_1-sn_2, K_2+sn_1)\mid s\in\bbZ, \text{ and }K_1-sn_2, K_2+sn_1>0\},
$$
so the set of all fractions
$$
\frac{n_1N_1}{n_1N_1+n_2N_2}=\frac{n_1N_1}{j-k_1-k_2},
$$
such that $(N_1,N_2)\in\mathcal N_j$, will consist of the real number $n_1N_1/(j-k_1-k_2)$ and all other rational numbers from $(0,1)$ that differ from the first number by an integer multiple of $\theta_j=n_1n_2/(j-k_1-k_2)$. 
Now since $\theta_j\to 0$ as $j\to\infty$, it follows that for every sufficiently large $j\in\bbN$, one can choose a pair $(N_{1,j},N_{2,j})\in\mathcal N_j$ so that~(\ref{N_j_conditions}) holds.
\end{proof}

\begin{proof}[Proof of Lemma~\ref{Inside_M_lemma}]
It was shown in~\cite{Firsova_Gor_equi} that $\partial\M\subset\mathcal X$, so we only need to show that the interior of $\M$ is contained in $\mathcal X$. 
Let $c_0\in \M$ be a non-critically periodic interior point of the Mandelbrot set. We note that $c_0$ belongs to either a hyperbolic or a queer component, in case if the latter ones exist. For each $k\in\bbN$, consider the map $F_k\colon\bbC\to\bbC$ defined by the formula
$$
F_k(c):= f_c^{\circ(k-1)}(c).
$$
Since $c_0\in\M$, the sequence $\{F_k(c_0)\}_{k=1}^\infty$ is bounded, hence, there exists a subsequence $\{k_m\}_{m\in\bbN}\subset\bbN$, such that the limit $\lim_{m\to\infty} F_{k_m}(c_0)$ exists and is equal to some number $w\in\bbC$. We may assume that $w\neq 0$. Otherwise, if $w=0$, then take the subsequence $\{k_m+1\}_{m\in\bbN}$ instead of the subsequence $\{k_m\}_{m\in\bbN}$. Since $c_0$ is an interior point of $\M$, the family of maps $\{F_{k_m}\}_{m\in\bbN}$ is normal, when restricted to some open neighborhood $U$ of $c_0$, so after further extracting a subsequence, we may assume that the sequence of functions $\{F_{k_m}\}_{m\in\bbN}$ converges to some holomorphic function $F\colon U\to\bbC$ on compact subsets of $U$. 

Let us assume that $c_0\not\in\mathcal X$. Then, according to Lemma~\ref{suff_cond_lemma}, there exists a closed half-plane $H\subset\bbC$, such that $0\in\partial H$ and for any repelling periodic orbit $\mathcal O\in\Omega_{c_0}$, we have $\nu_{\mathcal O}(c_0)\in H$. 
For any $z\in H$, let $\dist(z,\partial H)$ denote the Euclidean distance from $z$ to the boundary line $\partial H$ of $H$. 
Then under the above assumption, the following holds:

\begin{proposition}\label{assumption_prop}
Assume, the set $\M\setminus\mathcal X$ is nonempty and $c_0\in\M\setminus\mathcal X$. Let the half-plane $H$ and the sequence $\{k_m\}_{m\in\bbN}$ be the same as above. Then for any $\varepsilon>0$ there exists $M=M(\varepsilon)\in\bbN$, such that for any $m\ge M$ and any periodic orbit $\mathcal O\in\Omega_{c_0}$ of period $|\mathcal O|=k_m$, the inequality
$$
\dist(\nu_{\mathcal O}(c_0),\partial H)<\varepsilon.
$$
holds.
\end{proposition}
\begin{proof}
According to Lemma~\ref{summation_lemma}, 
\begin{equation*}
\lim_{m\to\infty} \sum_{j\in\bbN, j|k_m}\,\,\sum_{\mathcal O\in\Omega_{c_0}^j}\frac{j}{k_m}\nu_{\mathcal O}(c_0) = \lim_{m\to\infty} \frac{F_{k_m}'(c_0)}{k_mF_{k_m}(c_0)} = \lim_{m\to\infty} \frac{F'(c_0)}{k_mw}=0,
\end{equation*}
where $F$ is the limiting map of the sequence of maps $\{F_{k_m}\}_{m\in\bbN}$, and $w=F(c_0)\neq 0$. Note that for all but possibly one non-repelling orbit $\hat{\mathcal O}$ of fixed period $\hat j$, the terms in the above summation belong to $H$. As $k_m\to\infty$, the contribution $\frac{\hat j}{k_m}\nu_{\hat{\mathcal O}}$ of this non-repelling orbit in the summation goes to zero.
Then it follows that
\begin{equation*}
\lim_{m\to\infty} \dist\left(\sum_{\mathcal O\in\Omega_{c_0}^{k_m}}\nu_{\mathcal O}(c_0),\,\, \partial H\right) =0,
\end{equation*}
which implies Proposition~\ref{assumption_prop}.
\end{proof}

Finally, we complete the proof of Lemma~\ref{Inside_M_lemma} by observing that under the above assumption $c_0\not\in\mathcal X$, according to Lemma~\ref{suff_cond_lemma}, the half-plane $H$ can be chosen so that for at least one repelling periodic orbit $\mathcal O_1\in\Omega_{c_0}$, the value $\nu_{\mathcal O_1}(c_0)$ lies in the interior of $H$. Let $\mathcal O_2\in\Omega_{c_0}$ be any other repelling periodic orbit whose period is relatively prime to the period of $\mathcal O_1$. Then according to Lemma~\ref{averaging_lemma} and Proposition~\ref{averaging_extension_prop} with the parameter $\alpha$ fixed at $\alpha=1/2$, it follows that for each sufficiently large $m\in\bbN$, there exists a periodic orbit $\mathcal O\in\Omega_{c_0}$ of period $k_m$, such that 
$$
\dist(\nu_{\mathcal O}(c_0),\partial H) > \frac{1}{3}\dist(\nu_{\mathcal O_1}(c_0),\partial H).
$$ 
The latter contradicts to Proposition~\ref{assumption_prop}, hence the assumption $c_0\not\in\mathcal X$ was false. Since $c_0$ was an arbitrary non-critically periodic parameter from the interior of $\M$, and critically periodic parameters form a nowhere dense subset of $\M$, this completes the proof of Lemma~\ref{Inside_M_lemma}.
\end{proof}

\begin{proof}[Proof of Theorem~\ref{main_theorem_1}]
The proof is a combination of several lemmas: 
We have $\M\subset\mathcal X$ due to Lemma~\ref{Inside_M_lemma}. The set $\mathcal X$ is bounded according to Lemma~\ref{bounded_X_lemma}. From Lemma~\ref{connectedness_lemma} and Lemma~\ref{Inside_M_lemma} it follows that the set $\mathcal X$ is path connected. Finally Lemma~\ref{X_minus_M_nonempty_lemma} implies that the set $\mathcal X\setminus\M$ has nonempty interior.
\end{proof}

\appendix

\section{Pictures}\label{Pic_section}

\begin{figure}[h]
\begin{center}
\includegraphics[width=\textwidth]{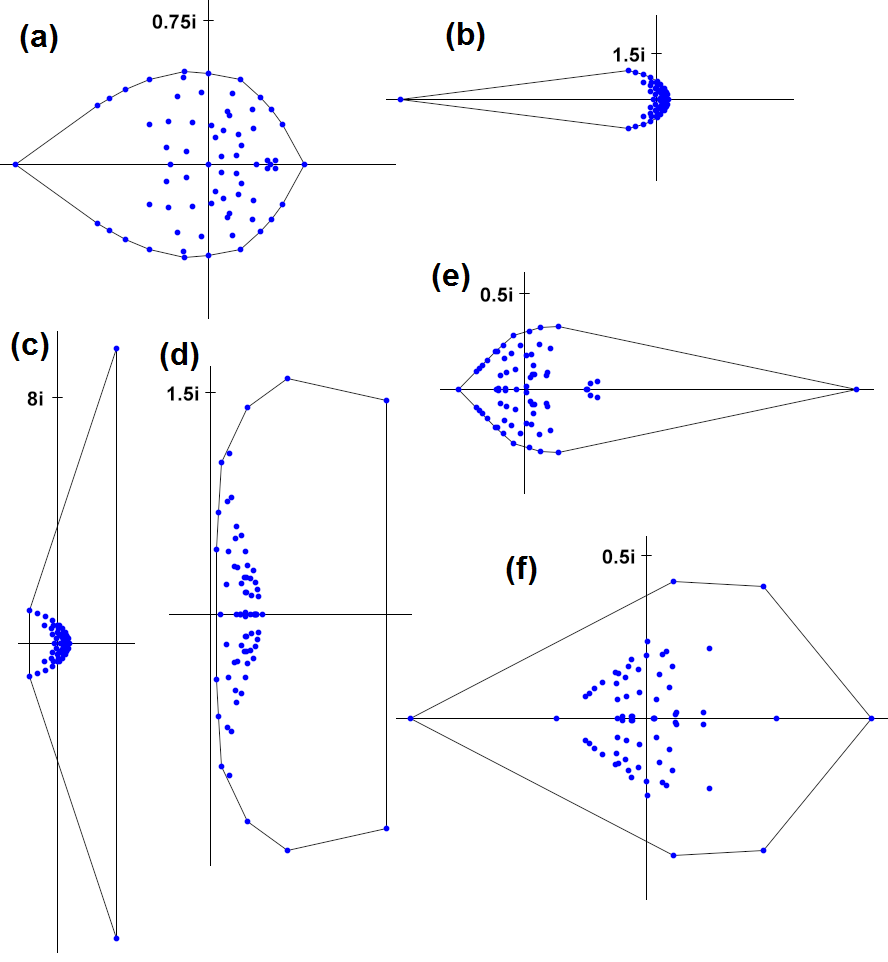}
\caption{Approximations of the sets $\mathcal Y_c$ as $c$ changes along the real axis: (a) $c=0$; (b) $c=0.24$; (c) $c=0.26$; (d) $c=0.42$; (e) $c=-0.71$; (f) $c=-1$.}\label{Y_c_pic}
\end{center}
\end{figure}

Theorem~\ref{main_theorem_1} and Theorem~\ref{main_theorem_2} provide efficient algorithms for constructing numerical approximations of the accumulation set $\mathcal X$ and the sets $\mathcal Y_c$.
For example, in order to approximate numerically the set $\mathcal X$, we first observe that according to Theorem~\ref{main_theorem_1}, the inclusion $\M\subset\mathcal X$ holds, so one only has to decide for each point $c\in\bbC\setminus\M$, whether it belongs to $\mathcal X$ or not. 
The latter can be done by means of Theorem~\ref{main_theorem_2} which provides an easy to check sufficient condition for $c\in\mathcal X$. More specifically, for each $c\in\bbC\setminus\M$, one should compute the points $\nu_{\mathcal O}(c)$, where $\mathcal O$ runs over different periodic orbits of the map $f_c$. If at some point $0$ falls into the convex hull of the computed points, then $c\in\mathcal X$. The periodic orbits of $f_c$ can in turn be computed by Newton's method (see~\cite{Hubbard_Schleicher_Sutherland} for a precise algorithm).

Figure~\ref{X_picture} is obtained by checking all periodic orbits of periods up to and including~8. The color of a point corresponds to the smallest period, up to which the periodic orbits need to be checked in order to confirm that $c\in\mathcal X$. The dark red strip in Figure~\ref{X_picture}, corresponding to period~8, is quite thin, so we hope that the picture gives a reasonably good approximation of the accumulation set $\mathcal X$ in Hausdorff metric, however, we don't know how to estimate the discrepancy. In particular, it is not clear whether the described algorithm can be used in order to numerically understand the fine structure of the boundary $\partial\mathcal X$.

Part~\ref{Y_convex_property} of Theorem~\ref{main_theorem_2} also allows to estimate numerically the sets $\mathcal Y_c$. Indeed, for any $c\in\bbC\setminus\{-2\}$, an approximation of $\mathcal Y_c$ can be constructed by taking the convex hull of the points $\nu_{\mathcal O}(c)$, where $\mathcal O$ runs over different periodic orbits of the map $f_c$. Figure~\ref{Y_c_pic} provides several pictures of the sets $\mathcal Y_c$, where the parameter $c$ takes different values on the real line. In particular, (a) and (f) correspond to the centers of the main cardioid and the hyperbolic component of period~2 respectively, and (b) and (c) correspond to the parameter $c$ lying slightly to the left and respectively slightly to the right of the cusp of the main cardioid. The blue dots are the values of $\nu_{\mathcal O}(c)$, for all repelling periodic orbits $\mathcal O$ of periods up to and including~8. We don't know, how accurate these pictures are, since inclusion of periodic orbits of higher periods can potentially change the convex hulls significantly.

\bibliographystyle{amsalpha}
\bibliography{biblio}

\end{document}